\documentclass[10pt,twoside]{article}

\usepackage{amsmath}
\usepackage{amssymb}
\usepackage{amsthm}
\usepackage[latin1]{inputenc}
\usepackage{eurosym}
\usepackage[dvips]{graphics}
\usepackage{graphicx}
\usepackage{epsfig}
\usepackage{hyperref}
\usepackage{ifthen}
\usepackage{setspace}
\usepackage[active]{srcltx}
\usepackage{pdfsync}
\usepackage{color}

\newcommand{\Rr}{{\mathbb{R}}}

\newcommand{\D}[1]{\mbox{\rm #1}}
\newcommand{\dd}{\D{d}}

\newtheorem{theorem}{Theorem}[section]
\newtheorem{corollary}[theorem]{Corollary}
\newtheorem{rem}[theorem]{\sc Remark}
    \newenvironment{remark}{\begin{rem} \begin{rm}}{\end{rm} \qed\end{rem}}
\newtheorem{lemma}[theorem]{Lemma}

 \newtheorem{proposition}[theorem]{Proposition}

\newtheorem{definition}[theorem]{Definition}

\theoremstyle{definition}

\numberwithin{equation}{section}

\makeatletter
\evensidemargin \oddsidemargin
\makeatother

\begin{document}
\thispagestyle{empty}
\setcounter{page}{1}



\begin{center}
{\large\bf \uppercase{On the Cauchy problem for a general fractional porous medium equation with variable density}}

\vskip.20in

Fabio Punzo \\[2mm]
{\footnotesize
Dipartimento di Matematica "G. Castelnuovo", \\
Universit\`{a} di Roma La Sapienza,\\
Piazzale Aldo Moro 5, 00185 Roma }\\[2mm]
Gabriele Terrone\footnote{
Gabriele Terrone was supported by the UTAustin-Portugal partnership through the FCT post-doctoral fellowship
SFRH/BPD/40338/2007, CAMGSD-LARSys through FCT Program POCTI -
FEDER and by grants PTDC/MAT/114397/2009,
UTAustin/MAT/0057/2008, and UTA-CMU/MAT/0007/2009.
} \\[2mm]
{\footnotesize
Center for Mathematical Analysis, Geometry, and Dynamical Systems,\\
Departamento de Matem\'{a}tica, \\
Instituto Superior T\'{e}cnico, 1049-001 Lisboa, Portugal}
\end{center}

{\footnotesize  
\noindent
{\bf Abstract.}
We study the well-posedness of the Cauchy problem for a fractional porous medium equation with a varying density $\rho>0$. We establish existence of weak energy solutions; uniqueness and nonuniqueness is studied as well, according with the behavior of $\rho$ at infinity.  \\[3pt]
}

\vskip.2in


\section{Introduction}
In this paper we study the following nonlinear nonlocal Cauchy problem:
\begin{equation}
\label{06111}
    \begin{cases}
   \rho\, \partial_t u + (- \Delta)^{\frac{\sigma}{2}}\left[ u^m\right] = 0  & x\in \Rr^N, \quad t>0\\
     u=u_0 & x\in \Rr^N, \quad t=0.
    \end{cases}
\end{equation}
The nonlocal operator $(- \Delta)^{\frac{\sigma}{2}}$ is the fractional Laplacian of order $\sigma/2$;  see for instance 
\cite{CafS} 
for a comprehensive account on the subject. The parameter $\sigma$ is supposed to vary in the open interval $(0,2)$, thus a representation for such operator in terms of a singular integral holds.  The function $\rho(x)$ is a density; it is assumed to be positive and to depend continuously on the spatial variable $x$. The initial value $u_0$ is a bounded function belonging to the weighted space $L^1_{\rho}(\Rr^N)$ of measurable functions $f$ satisfying
$\int_{\Rr^N} f\,\rho\,  \dd x <\infty$. Finally, $N\geq 1 $ and $m$ is a real parameter greater or equal to $1$.
The aim of this paper is to investigate existence and uniqueness of solutions to problem \eqref{06111}.

\bigskip
By replacing the nonlocal operator in \eqref{06111} with the classical Laplace operator $\Delta$ we obtain the initial value problem for the porous medium equation with variable density:
\begin{equation}
\label{16011}
    \begin{cases}
    \displaystyle\rho\, \partial_t u - \Delta u^m  = 0  & x\in \Rr^N, \, t>0\\
     u=u_0 & x\in \Rr^N, \, t=0.
    \end{cases}
\end{equation}
Problem \eqref{16011} have been extensively studied in the literature; see \cite{Eid90}, \cite{EK},  \cite{GMP}, \cite{GHP}, \cite{KRV}, \cite{KKT}, \cite{P1}, and also  \cite{P2}-\cite{P4} where similar problems on Riemannian manifolds have been taken into account. The picture for problem \eqref{16011} has been completed in \cite{RV1}-\cite{RV2} where  existence and uniqueness of solutions to this problem have been established in the class of finite energy solutions assuming the initial data $u_0$ nonnegative and in $L^1_\rho$. 
Uniqueness of solutions to \eqref{16011} is a delicate issue and is strictly related with the behavior at infinity of the density $\rho$. More precisely, if $N=1$ or $N=2$, then uniqueness of solutions holds if $\rho$ merely belongs to $L^\infty$ (see \cite{GHP}). If instead $N\geq 3$ an additional requirement on $\rho$ must be satisfied in order to get uniqueness, namely that $\rho(x)$ vanishes {\em slowly}  as $|x|$ diverges, whereas nonuniqueness phenomena arise if the opposite behavior is satisfied by $\rho$ (see \cite{EK}, \cite{GMP}, \cite{KKT},  \cite{P1}-\cite{P4}, \cite{RV1}-\cite{RV2}). 

\bigskip
If $\rho\equiv 1$ in \eqref{06111}, we get the following nonlocal version of the initial value problem for the porous medium equation:
\begin{equation}
\label{16012}
    \begin{cases}
    \partial_t u + (- \Delta)^{\frac{\sigma}{2}}\left[ u^m\right] = 0  & x\in \Rr^N, \, t>0\\
     u=u_0 & x\in \Rr^N, \, t=0.
    \end{cases}
\end{equation}
This problem has been studied very recently in \cite{V2} where existence,
uniqueness and properties of weak solutions to \eqref{16012} have
been established assuming $u_0\in L^1(\Rr^N)$; the particular case $\sigma=1$
has been addressed in \cite{V1}.

\bigskip
The study of problem \eqref{06111} makes perfect sense, as it can be regarded both as a generalization of problem \eqref{16012} and as nonlocal version of problem \eqref{16011}. Moreover, problem \eqref{06111} arises in many physical situations (see, $e.g.$ \cite{AC}, \cite{Ja} , \cite{JKO}) such as diffusions in inhomogeneous media and is particularly interesting from a probabilistic point of view since, as  it is well known, the fractional Laplacian is the infinitesimal generator of a L\'evy process (see \cite{Bert}). Nonetheless, to the best of our knowledge, the analysis of problems like \eqref{06111} is relatively new in the literature. Some results for nonlocal linear parabolic equation with a variable density have been established in \cite{Chas}, but not for problem considered in this paper. Recently, in \cite{PT1}, it has been studied the special case $N=\sigma=1$, that is 
\begin{equation}
\label{16014}
    \begin{cases}
    \displaystyle\rho\, \partial_t u + \left(-\frac{\partial^2}{\partial x^2}\right)^{\frac{1}{2}}\left[ u^m\right] = 0  & x\in \Rr, \, t>0\\
     u=u_0 & x\in \Rr, \, t=0.
    \end{cases}
\end{equation}
In the light of results in \cite{EK}, \cite{GMP} \cite{KKT}, and \cite{P1}, bounded initial data have been considered in \cite{PT1}; existence and uniqueness of {\em very weak solutions} to problem \eqref{16014} (namely solutions not having finite energy in the whole $\Rr^N$) have been proved in the class of bounded solutions not satisfying any extra conditions at infinity. 

We point out that the arguments used in \cite{PT1} are completely different from those in the present paper.
In fact, as well as in \cite{KRV}, \cite{RV1}-\cite{RV2}, \cite{V1}-\cite{V2},
we deal here with {\it weak energy solutions} to problem \eqref{06111} (see Definition \ref{defsol}), and consider nonnegative bounded initial data $u_0$ belonging to $L^1_{\rho}(\Rr^N)$. We emphasize also that  our results differ from those in \cite{V1}-\cite{V2}, where $\rho$ is constant.

\bigskip
We outline next the structure and main contributions of this paper. In Section \ref{sec:aa} after recalling the mathematical background about the fractional Laplacian, such as its realization through the harmonic extension, we give the precise notion of solution we will considered. In Section \ref{existence} we prove existence of weak energy solutions. The presence of the varying density $\rho$ does not bring any additional technical difficulty at this stage of the work, and the proof of the main result of this Section, Theorem \ref{02011}, goes along the same lines as the proof of the existence results in \cite{V1}, \cite{V2} and \cite{PT1}; however we will sketch it for seek of completeness and for further references. 

In Sections \ref{slow} and \ref{fast} we deal with uniqueness and nonuniqueness of solutions. Concerning these issues, as expected, problem \eqref{06111} turns out to share many aspects with its local counterpart, problem \eqref{16011}.  

First, in Theorem \ref{tuni} we establish uniqueness under the additional requirement that $\rho(x)$ vanishes {\em slowly} as $|x|$ diverges. As a byproduct, we show that total mass is conserved along the evolution; see Proposition \ref{propcm}.

The opposite situation in which $\rho(x)$ dacays {\em fast} as $|x|\to \infty$ is studied in Section \ref{fast}. We first prove in Theorem \ref{texi2} that in this case there exist solutions to \eqref{06111} satisfying an extra condition at infinity (see \eqref{ea21}); the proof of this results makes use of a Theorem shown in \cite{Rubin} and requires $N\geq 2$. As a consequence, in Corollaries \ref{cornu1} and \ref{cornu2} we easily obtain nonuniqueness of bounded solutions, if we do not specify their behavior at infinity. Instead, we shall prove that uniqueness is restored in the class of solutions satisfying a proper decay condition at infinity (see Theorem \ref{tunc}).  

Finally, in Section \ref{half} we study the particular situation in which $\sigma = 1$ in \eqref{06111}, that is:
\begin{equation}
\label{06111a}
    \begin{cases}
   \rho\, \partial_t u + (- \Delta)^{\frac{1}{2}}\left[ u^m\right] = 0  & x\in \Rr^N, \quad t>0\\
     u=u_0 & x\in \Rr^N, \quad t=0.
    \end{cases}
\end{equation}
In this case, it is possible to get rid of the boundedness assumption on the initial data and to generalize to the case $u_0\in L^1_{\rho}(\Rr^N)$ existence and uniqueness results previously discussed for $u_0\in L^1_{\rho}(\Rr^N)\cap L^\infty(\Rr^N)$. A key tool, for this scope, is a smoothing estimate (see Theorem \ref{tse}), which holds true under the requirement $\rho\in L^\infty(\Rr^N)$, that we prove by slightly adapting an argument of \cite{V1}. 
\bigskip
Let us mention that, to the best of our knowledge, our results are new also in the linear case ($m=1$), that is for problem 
\[
    \begin{cases}
   \rho\, \partial_t u + (- \Delta)^{\frac{\sigma}{2}}\left[ u\right] = 0  & x\in \Rr^N, \quad t>0\\
     u=u_0 & x\in \Rr^N, \quad t=0.
    \end{cases}
\]
In this case, losely speaking, uniqueness of solutions  corresponds to the fact that the L\'evy process associated to the operator $\frac 1{\rho} (- \Delta)^{\frac{\sigma}{2}}$, starting from any point in $\Rr^N$, does not attain {\it infinity}. On the contrary, a solution satisfying additional conditions at infinity exists when the L\'evy process exits arbitrarily large balls.


\section{Mathematical background}
\label{sec:aa}

The fractional Laplacian $(-\Delta)^{\sigma/2}$ is a nonlocal partial differential operator; it can be defined in many different ways, one of which relies on the Fourier transform. For any $g$ in the class of Schwartz functions, if $(-\Delta)^{\sigma/2} g = h$ then

\begin{equation}
\label{02015}
\hat{h}(\xi)= |\xi|^{\sigma}\hat{g}(\xi).
\end{equation}
If we require $\sigma$ to vary in the open interval $(0,2)$ we can use the representation

\begin{equation}
\label{ea1}
(-\Delta)^{\sigma/2} g(x)=C_{N,\sigma}\,\, \textrm{P.V.}\, \int_{\Rr^N} \frac{g(x)-g(z)}{|x-z|^{N+\sigma}}\dd z,
\end{equation}
where $C_{N,\sigma}=\frac{2^{\sigma-1}\sigma \Gamma((N+\sigma)/2)}{\pi^{N/2}\Gamma(1-\sigma/2)}
$ is an appropriate positive normalization constant depending on $N$ and $\sigma$.

\smallskip

In the following Sections \ref{existence}, \ref{slow} and \ref{fast} we will assume:
\begin{equation}
\label{A0} \tag{{\bf A}$_0$}
\begin{cases}
\text{(i)} &\rho\in C(\Rr^N), \,\rho>0 \text{ in } \Rr^N,\\
\text{(ii)}& u_0\in L^{\infty}(\Rr^N) \cap L^+_\rho(\Rr^N),\\
\text{(iii)} &m\ge 1,\\
\text{(iv)} &0<\sigma <2.
\end{cases}
\end{equation}
Here
\begin{gather*}
L^1_\rho(\Rr^N):= \left\{ f \text{ measurable in }\Rr^N\, \Big{|}\, \|f\|_{L^1_{\rho}(\Rr^N)}:=\int_{\Rr^N} f \, \rho\, \dd x <\infty\right\},\\
L^+_\rho(\Rr^N):= \left\{ f\in L^1_\rho(\Rr^N)\, \Big{|}\, f\geq 0\right\}.\\
\end{gather*}
In the final Section \ref{half} we will modify Assumption \eqref{A0} by requiring $u_0\in L^+_\rho(\Rr^N)$ and $\rho \in L^\infty(\Rr^N)$.  

Observe that in view of Assumption \eqref{A0}--(i), the measure $\rho(x) \dd x$ is absolutely continuous with respect the Lebesgue measure in $\Rr^N$ and $\displaystyle\lim_{q\to\infty}\|f\|_{L^q_\rho}=\|f\|_{\infty}$; here $\|f\|_{L^q_{\rho}}$ is the weighted norm $\left(\int_{\Rr^N} |f|^q \, \rho\, \dd x\right)^{1/q}$ while $\|\cdot\|_{\infty}$ is the usual norm in $L^\infty(\Rr^N)$. Notice also that Assumption \eqref{A0}--(i) implies $L^1(\mathcal O) = L^1_\rho (\mathcal O)$ for any bounded domain $\mathcal O \subset \Rr^N$.

\smallskip

Multiplying the nonlocal partial differential equation in \eqref{06111} by a test function $\psi$ compactly supported in $\Rr^N\times (0,T)$, $T>0$, integrating by parts, taking into account \eqref{02015} and using the Plancherel's Theorem, we discover that
\begin{equation}
\label{02016}
\int_0^T \int_{\Rr^N} \rho\, u\, \partial_t \psi \dd x\, \dd t - \int_0^T \int_{\Rr^N}(-\Delta)^{\sigma/4}(u^m)\,
(-\Delta)^{\sigma/4}\psi \, \dd x\, \dd t = 0.
\end{equation}

The integrals above make sense if the function $u^m$ belongs to an appropriate space, namely the fractional Sobolev space $\dot{H}^{\sigma/2}(\Rr^N)$, which is the completion of $C^\infty_0(\Rr^N)$ with the norm
$\|\psi\|_{\dot{H}^{\sigma/2}}= \|(-\Delta)^{\sigma/4} \psi\|_{L^2(\Rr^N)}$.

\begin{definition}\label{defsol}
A \emph{solution} to problem \eqref{06111} is a function $u\geq 0$ such that:
\begin{itemize}
\item{} $u\in C([0,\infty); L^1_{\rho}(\Rr^N))$ and $u^m \in L^2_\text{loc}((0, \infty): \dot{H}^{\sigma/2}(\Rr^N))$;
\item for any $T>0$, $\psi\in C^{1}_0(\Rr^N \times (0, T))$ identity \eqref{02016} holds;
\item $u(\cdot, 0)=u_0$ almost everywhere.
\end{itemize}
\end{definition}
In accordance with the terminology of \cite{V2}, such solutions can be called $L^1_{\rho}$ {\it weak energy solutions}. 

\bigskip

If $\varphi$ is a smooth and
bounded function defined in $\Rr^N$, we can consider its $\sigma$--harmonic
extension $v=\D{E}(\varphi)$ to the upper half-space
\[
\Omega:= \Rr^{N+1}_+ = \{(x,y): \, x\in \Rr^N, \, y>0 \},
\]
that is, the unique smooth and bounded solution $v(x,y)$ of the problem
\[
  \begin{cases}
    \nabla \left(y^{1-\sigma}\nabla v\right)=0 &\text{in }\Omega\\
    v(x,0)=\varphi(x) & \text{in }\Gamma.
    \end{cases}
\]
Here $\Gamma:=\overline \Omega \cap \{y=0\}\equiv \Rr^N$. It has been proved (see 
\cite{CafS},  \cite{V2}) that 
\[
-\mu_\sigma \lim_{y\to 0^+} y^{1-\sigma}\frac{\partial v}{\partial y}=\left(-\Delta \right)^{\frac \sigma 2} \varphi(x)\quad 
\textrm{for all } x\in \Gamma,
\]
where $\mu_\sigma:=\frac{2^{\sigma-1}\Gamma(\sigma/2)}{\Gamma(1-\sigma/2)}$. We then define the operators
\begin{gather*}
L_\sigma v := \nabla \left(y^{1-\sigma}\nabla v\right)\\
\frac{\partial v}{\partial y^\sigma}:= \mu_\sigma \lim_{y\to 0^+} y^{1-\sigma}\frac{\partial v}{\partial y}.
\end{gather*}

Solving problem \eqref{06111} is equivalent to solving the following  quasi-stationary problem for $w=\D{E}(u^m)$, with a dynamical boundary conditions:
\begin{equation}
\label{06112}
    \begin{cases}
    L_\sigma w = 0 & (x,y)\in \Omega, \, t >0\\
    \displaystyle \frac{\partial w}{\partial y^\sigma}= \rho\, \frac{\partial\left[ w^{\frac{1}{m}}\right]}{\partial t} & x\in \Gamma,\, t>0\\
     w=u_0^{m} & x\in \Gamma, \, t=0.
    \end{cases}
\end{equation}

We introduce next weak energy solutions of problem \eqref{06112}. Formally, multiplying the differential equation in \eqref{06112} by a test function $\psi$ compactly supported in $\bar\Omega\times (0, T]$, integrating by parts and taking into account  initial condition and the dynamical boundary condition  we get:

\begin{equation}
\frac{1}{\mu_\sigma}\int_0^T \int_{ \Gamma} \rho\, u\, \partial_t \psi \,\dd x\,  \dd t
= \int_0^T \int_{\Omega} \nabla\psi \cdot \left(y^{1-\sigma}\nabla w\right) \,\dd x\, \dd y\, \dd t\,. \label{291101}
\end{equation}

We denote by $X^\sigma(\Omega)$ the completion of $C^\infty_0(\Omega)$ with the norm
\[
\|v\|_{X^\sigma} = \left(\mu_\sigma \int_{\Omega} y^{1-\sigma}|\nabla v|^2\right)^\frac{1}{2}.
\]
Given a function $f\in W^{1,2}(\Omega)$ we denote by $f |_{\Gamma}$ its trace on $\Gamma$, which is in $L^2(\Gamma)$.
\begin{definition}
\label{06114}
A \emph{solution} to problem \eqref{06112} is a
pair of functions $(u,w)$ with $u\geq 0$, $w \geq 0$, such that
\begin{itemize}
\item{}   $u \in C\big([0, \infty); L^{1}_{\rho} (\Gamma)\big)$;
\item{} $ w\in L^2_{loc}\big((0,\infty); X^\sigma(\Omega)\big)$;
\item{} $w|_{\Gamma\times (0,\infty)}=u^m$;
\item{} for any $T>0, \psi\in C^1_0(\bar\Omega \times (0, T))$ there holds
 \begin{equation}
    \label{06115}
  \int_0^T \int_{ \Gamma} \rho\, u\, \partial_t \psi \,\dd x\,  \dd t =      {\mu_\sigma}  \int_0^T \int_{\Omega} \nabla\psi \cdot \left(y^{1-\sigma}\nabla w\right) \,\dd x\, \dd y\, \dd t;
    \end{equation}
\item{} the identity $u(\cdot, 0)=u_0$ holds almost everywhere.
\end{itemize}
\end{definition}

The following result establishes the equivalence between the two notions of solutions given in Definition \ref{defsol} and Definition \ref{06114}; it can be proved as in \cite[Section 3.3]{V2}.
\begin{proposition}
\label{02012}
A function $u$ is a solution to problem \eqref{06111} if and only if $(u, \D{E}(u^m))$ is a solution to problem \eqref{06112}.
\end{proposition}


\section{Existence of solutions}
\label{existence}
The aim of this Section is to establish the following

\begin{theorem}
\label{02011}
Let assumption \eqref{A0} be satisfied. Then there exists a solution $(u,w)$ to problem \eqref{06112}.  Furthermore, 
\begin{equation}
\label{ea40}
\begin{split}
\|u\|_{L^\infty(\Rr^N\times (0,\infty))}\leq \|u_0\|_{L^\infty(\Rr^N)}, \\
\|w\|_{L^\infty(\Rr^N\times (0,\infty))}\leq \|u_0^m\|_{L^\infty(\Rr^N)}
\end{split}
\end{equation}
and
\begin{equation}
\label{ea40bb} 
\|u(\cdot, t)\|_{L^1_{\rho}}\leq \|u_0\|_{L^1_\rho}\; \;  \textrm{for any } t>0\,. 
\end{equation}
\end{theorem}

\begin{remark}\label{solmin} In the proof of Theorem \ref{02011} a solution $(u,w)$ is constructed. Such solution turns out to be  {\it minimal}, in the sense that if  $(\tilde u,\tilde w)$  is another solution, then $u\leq \tilde u$ and $w\leq \tilde w$.
\end{remark}

\begin{proposition}\label{propcontraz}
Let assumption \eqref{A0} be satisfied. Let $(u,w)$ and $(\hat u, \hat w)$ be minimal solutions to problem \eqref{06112} provided by Theorem \ref{02011}, corresponding to initial data $u_0$ and $\hat u_0$, respectively. Then, for any $t>0$,
\begin{equation}\label{e40b}
\int_{\Gamma} \left[ u (x,t)- \hat{u}(x,t))\right]_{+} \rho\, \dd x
\leq \int_{\Gamma} \left[ u_0 - \hat{u}_0 \right]_{+} \rho\, \dd x\,.
\end{equation}
 \end{proposition}

The statement  of Theorem \ref{02011} and Proposition \ref{propcontraz} above can be proved proceeding as in \cite{V2} and \cite{PT1}. In fact, at this stage of the analysis, the presence of the varying density $\rho$ does not bring any additional difficulty. However, we sketch the main steps of their proofs for seek of completeness and for later references.

\bigskip

We use next a spectral decomposition to define the fractional operator $(-\Delta)^{\sigma/2}$ in a bounded domain $\mathcal O$ of $\Rr^N$. In fact, let $\{\xi_n\}_1^\infty$  be an orthonormal basis of $L^2(\mathcal O)$  made by eigenfunctions of $-\Delta$ in $\mathcal O$ completed with homogeneous Dirichlet boundary conditions, and let $\{\lambda_n\}_1^\infty$
the sequence of the corresponding eigenvalues. For any $u\in C^{\infty}_0(\mathcal O)$ 
\[ (-\Delta)^{\sigma/2}u := \sum_{n=1}^\infty \lambda_n^{\sigma/2} \, u_n \, \xi_n\quad \textrm{in } \mathcal O\,, \]
where $u=\sum_{n=1}^\infty u_n\,  \xi_n$ in $L^2(\mathcal O).$
By density, $(-\Delta)^{\sigma/2} u$ can be also defined for $u$ belonging to the Hilbert space
\[ 
H_0^{\sigma/2}(\mathcal O):= \left\{u\in L^2(\mathcal O)\,\,\Big{|}\,\, \|u\|^2_{H_0^{\sigma/2}}:=\sum_{n=1}^\infty \lambda_n^{\sigma/2}u_n^2 <\infty \right\}.
\]

Recall that for every $R>0$, in view of hypothesis \eqref{A0}--(i), we have $L^1(B_R)\equiv L^1_{\rho}(B_R)$, $B_R$ being the open ball of radius $R$ with center at $0$.

Let $R>0$, $u_0\in L^1_{\rho}(B_R)$. We consider the following Cauchy-Dirichlet problem in $B_R$:
\begin{equation}
\label{02013}
    \begin{cases}
   \rho\, \partial_t u + (- \Delta)^{\frac{\sigma}{2}}\left[ u^m\right] = 0  & x\in B_R, \, t>0,\\
     u=0 & x\in \partial B_R, \, t>0,\\
     u=u_0 &  x\in B_R, \, t=0
     \end{cases}
\end{equation}
and  give next
\begin{definition}
\label{02017}
A \emph{solution} to problem \eqref{02013} is a function $u\geq 0$ such that:
\begin{itemize}
\item{} $u\in C([0,\infty); L^1_{\rho}(B_R))$ and $u^m \in L^2_\text{loc}((0, \infty): \dot{H}^{\sigma/2}(B_R))$;
\item for any $T>0, \psi\in C^1_0(B_R \times (0,T))$ there holds
\begin{equation}
\label{02018}
\int_0^T \int_{B_R} \rho\, u\, \partial_t \psi \,\dd x\, \dd t = \int_0^T \int_{B_R}(-\Delta)^{\sigma/4}u^m\,
(-\Delta)^{\sigma/4}\psi \, \dd x\, \dd t;
\end{equation}
\item $u(\cdot, 0)=u_0$ almost everywhere in $B_R$.
\end{itemize}
\end{definition}

As well as for problem \eqref{06111}, to solve problem \eqref{02013} we can also consider the analogous of problem \eqref{06112} in the half-cylinder $\mathcal{C}_R:= B_R\times (0, \infty)$ with zero lateral condition:
\begin{equation}
\label{02019}
    \begin{cases}
    L_\sigma w = 0 & (x,y)\in \mathcal{C}_R, \, t >0;\\
    w=0 & x\in \partial\mathcal{C}_R,\, y>0,\, t>0;\\
    \displaystyle \frac{\partial w}{\partial y^\sigma}= \rho\, \frac{\partial\left[ |w|^{\frac{1}{m}} \right]}{\partial t} & x\in \mathcal{C}_R,\, y=0,\, t>0;\\
     w=u_0^m & x\in\mathcal{C}_R, y=0,\, t=0.
    \end{cases}
\end{equation}

\begin{definition}
\label{020110}
A \emph{solution} to problem \eqref{02019} is a pair of functions $(u,w)$, with $u\geq 0$, $w\geq 0$, such that:
\begin{itemize}
\item{}   $u \in C\big([0, \infty); L^{1} _{\rho}(B_R)\big)$;
\item{} $ w\in L^2_{loc}\big((0,\infty); X_0^\sigma(\mathcal{C}_R)\big)$;
\item{} $w|_{B_R\times (0,\infty)}=u^m$;
\item{} for any $T>0$ and $\psi=\psi(x,y,t)$, $\psi\in C^1_0(B_R \times [0, \infty) \times (0, T))$,   there holds
 \begin{equation}
    \label{03011}
      \int_0^T \int_{ B_R} \rho\, u\, \partial_t \psi \,\dd x\,  \dd t =  {\mu_\sigma}\int_0^T \int_{\mathcal{C}_R} \nabla\psi \cdot \left(y^{1-\sigma}\nabla w\right) \,\dd x\, \dd y\, \dd t;
    \end{equation}
\item{} the identity $u(\cdot, 0)=u_0$ holds almost everywhere in $B_R$.
\end{itemize}
\end{definition}
As well as 
in the case of 
$\Rr^N$ (see Proposition \ref{02012}), the two notions of solutions given in Definition \ref{02017} and Definition \ref{020110} are equivalent.

\bigskip
The following existence result holds for problem \eqref{02019}. 

\begin{proposition}
\label{03012}
Let assumption \eqref{A0} be satisfied. Then for any $R>0$ there exists a 
solution $(u_R,w_R)$ to problem \eqref{02019}. 
Moreover, the following properties are satisfied:
\begin{itemize}
\item[i.] If $u_R$ and $\tilde{u}_R$ are solutions of \eqref{02019} corresponding to initial data $u_0$ and $\tilde{u}_0$ respectively, then
\[
\int_{B_R} \left[ u_R (x,t)- \tilde{u}_R(x,t))\right]_{+} \rho\, \dd x
\leq \int_{B_R} \left[ u_0 - \tilde{u}_0 \right]_{+} \rho\, \dd x;
\]
in particular, since
$u_0\geq 0$ then $u_R(x,t)\geq 0$ for every
$x\in B_R$ and every $t>0$;
\item[ii.]  $0\leq w_R \leq \|u_0\|_{\infty}^m$ for every $x\in \mathcal C_R$ and every $t>0$, $0\leq u_R\leq \|u_0\|_{\infty}$ for every $x\in B_R$ and every $t>0$;
\item[iii.] For any $R'>R$, $w_{R'}(x,t) \geq w_{R}(x,t)$ for every $x\in \mathcal C_{R}$ and every $t>0$;
\end{itemize}
\end{proposition}

\begin{proof}
The statement follows as well as in \cite[Theorem 7.1 and Theorem 7.2]{V2} and \cite[Proposition 3.5]{PT1}.
\end{proof}

\begin{proof}
[Proof of Theorem \ref{02011}.]
To solve the problem in the whole $\Rr^N$ we proceed along the same lines of \cite[Theorem 7.2]{V2} and \cite[Theorem 3.1]{PT1}. 
For any $R>0$, by Proposition \ref{03012} there exists a weak solution $(u_R, w_R)$ to problem \eqref{02019} in ${\mathcal C_R}\times (0,\infty)$. Since $w_R$ are monotonic decreasing with respect to $R$ and uniformly bounded, there exist the limits
 \begin{gather*}
 \lim_{R\to \infty} u_R=: u \quad\text{in $\Gamma\times (0,\infty)$},\\
 \lim_{R\to \infty}w_R =: w \quad\text{in $\Omega\times (0,\infty)$}.
 \end{gather*}
Then by usual compactness arguments, it is easy to check that $(u,w)$ is a solution to \eqref{06112}. Clearly, by construction, $(u,w)$  is the minimal solution.  
Moreover, from Proposition \ref{03012} we can infer that \eqref{ea40} and \eqref{ea40bb} hold true.
\end{proof}

We conclude this Section by showing the $L^1_{\rho}$ contraction principle stated in Proposition \ref{propcontraz}.
\begin{proof}[Proof of Proposition \ref{propcontraz}]
Arguing as in the proof of Theorem \ref{02011}, we have 
\[ 
u = \lim_{R\to \infty}u_R,  \qquad \hat u= \lim_{R\to \infty} \hat u_R  \qquad\textrm{in }\Rr^N\,,
\]
where $u_R$ and $\hat u_R$ solve the approximating problem \eqref{02019} with initial data $u_0$ and $\hat u_0$, respectively.  
Observe that 
\begin{gather*}
u_0\chi_{B_R} \to u_0, \quad \hat u_0\chi_{B_R}\to \hat u_0 \quad \textrm{as } R\to\infty, \quad\textrm{in }\Rr^N,\\
 |u_0-\hat u_0|\chi_{B_R}\leq |u_0|+|\hat u_0|\quad \textrm{in }\Rr^N, \quad  \textrm{for all }R>0.
\end{gather*}
Furthermore, 
\[ 
|u_R-\hat u_R| \leq |u|+|\hat u|   \quad \textrm{in } B_R, \quad\textrm{for all }R>0.
\]
Since $u_0, \hat u_0, u, \hat u\in L^+_\rho(\Rr^N)$, from the dominated convergence theorem and Proposition \ref{03012}-$i.$ we obtain the conclusion.
\end{proof}

Henceforth, unless otherwise specified, the term {\em solution} must  be understood in the sense of Definition \ref{defsol}.

\section{Slowly decaying density}
\label{slow}
Let us assume the following condition:
\begin{equation}
\label{A1}
\tag{{\bf A}$_1$}
\begin{array}{c}
  \textrm{there exist}\; \hat C>0, \hat R>0 \; \textrm{and}\; \alpha\in (0, \sigma) \; \textrm{such that}\\
  \rho(x)\geq \hat C  |x|^{- \alpha}\quad \textrm{for all}\;\, x\in \Rr^N\setminus B_{\hat R}\,\,.
\end{array}
\end{equation}
Under hypothesis \eqref{A1} we will establish both uniqueness of solutions not satisfying any extra condition at infinity
and conservation of mass. 

\subsection{Uniqueness of solutions}
We shall prove the following
\begin{theorem}\label{tuni}
Let assumptions \eqref{A0}, \eqref{A1} be satisfied. Then there exists at
most only one bounded solution to problem \eqref{06111}.
\end{theorem}

Before proving Theorem \ref{tuni}, let us introduce the following notations. First of all we set, for later use,
\begin{equation}
\label{ea22}
G(s):=s^m \quad (s\in\Rr^+).
\end{equation} 
We also take a nonnegative non-increasing cut-off function $\eta$ such that
\begin{equation}\label{ea5}
\eta(s)=
\begin{cases}
1  &\textrm{if } 0\leq s\leq 1,\\ \\
0 & \textrm{if } s\geq 2;
\end{cases}
\end{equation}
then, for each $R>0$, define
\begin{equation}\label{ea2}
\varphi_R(x):= \eta\left(\frac{|x|}{R} \right) \;\;  \textrm{for all}\;\;  x\in\Rr^N\,.
\end{equation}
Observe that, for  $R=1$, $\varphi_1(x)=\eta(|x|)\;\; (x\in \Rr^N)$.

By using the representation \eqref{ea1} it can be shown the following lemma (see \cite[Section 9.2]{V2}, \cite{BonfV}).
\begin{lemma} 
For any $R>0$ let $\varphi_R$ be the function defined by \eqref{ea2} and \eqref{ea5} and let $y:=\frac{x}R$. Then, for any $R>0$,
\begin{equation}
\label{ea3}
(-\Delta)_x^{\sigma/2} \varphi_R (x) = R^{-\sigma}  (-\Delta)_y^{\sigma/2}\varphi_1(y) \quad \textrm{for all}\;\; x\in\Rr^N.
\end{equation}
Furthermore, there exists a constant $\bar C>0$ such that
\begin{equation}\label{ea4}
\left|  (-\Delta )^{\sigma/2}\varphi_1(x)  \right| \leq \frac{\bar C}{1+|x|^{N+\sigma}}\quad \textrm{for all}\;\; x\in \Rr^N\,.
\end{equation}
\end{lemma}

\medskip

We are now in position to prove Theorem \ref{tuni}.

\begin{proof}
[Proof of Theorem \ref{tuni}]
Suppose, by contradiction, that there exist two different solutions $u, \tilde
u$ to problem \eqref{06111}. Take any $p\geq 1$ and $q\geq 1$ to be fixed later such
that $\frac 1 p + \frac 1 q=1$. 

Since $u$ and $\tilde u$ are bounded, we get, for some $L>0$,
\begin{equation}\label{ea15}
|G(u)- G(\tilde u)|\le L |u-\tilde u|\le C |u-\tilde u|^{\frac 1
q}\quad \textrm{in}\;\; \Rr^N\times (0,\infty)\,,
\end{equation}
where $C:=L\big(\|u\|_{\infty} + \|\tilde u\|_{\infty}
\big)^{1-\frac1 q}$ and $G$ is defined in \eqref{ea22}.

\smallskip

Let $0<t_1<t_2<T$ and set $n_0:=\left[\frac{2}{t_2-t_1}\right]+1$. Consider a sequence of functions
$\{f_n\}_{n\ge n_0}\subset C^{\infty}([0,\infty))$,  such that, for any $n\geq n_0$,
\begin{gather*}
0\leq f_n(t)\leq 1\textrm{ for any }t\geq 0, \qquad f_n(t)=0\textrm{ for any }t\notin [t_1, t_2], \\
f_n(t)=1 \textrm{ for any }t\in \left[t_1+\frac 1 n, t_2-\frac 1 n \right].
\end{gather*}
Note that
$f_n(t) \to \chi_{[t_1, t_2]}(t)$ as $n\to \infty$ for any $t\geq 0$,
so
$ f'_n(t)$ converges to $\delta(t-t_1) - \delta(t-t_2)$ as $n\to \infty$ in the sense of distributions.

\smallskip

Let $\psi\in C^{\infty}_0(\Rr^N);$ for any $n\geq n_0$ set
\[ 
\psi_n(x,t):= \psi(x) f_n(t) \quad\text{for all } x\in\Rr^n, \, t\geq 0.
\]
For each $n\geq n_0$ we can use such $\psi_n$ as test function in Definition \ref{defsol}, so, integrating by parts, 
\begin{align*}
\int_0^T \int_{\Rr^N} \rho\, u\,\psi f'_n(t) \dd x\, \dd t &= \int_0^T \int_{\Rr^N}f_n(t)(-\Delta)^{\sigma/4}(G(u))\,(-\Delta)^{\sigma/4}\psi \, \dd x\, \dd t\\
 &= - \int_0^T \int_{\Rr^N} f_n(t) G(u) (-\Delta)^{\sigma/2}  \psi \, \dd x \,\dd t. 
\end{align*}
Then sending $n\to \infty$ we get:
\begin{equation}\label{ea6}
\int_{\Rr^N} \rho(x)[ u(x, t_2)- u(x, t_1) ] \psi(x) \dd x \,=\,
\int_{t_1}^{t_2}\int_{\Rr^N} G(u) (-\Delta)^{\sigma/2}  \psi \dd x
\dd t\,.
\end{equation}
Analogously we have
\begin{equation}\label{ea7}
\int_{\Rr^N} \rho(x)[ \tilde u(x, t_2)- \tilde u(x, t_1) ] \psi(x)
\dd x \,=\, \int_{t_1}^{t_2}\int_{\Rr^N} G(\tilde u)
(-\Delta)^{\sigma/2}  \psi \dd x \dd t \,.
\end{equation}
Subtracting \eqref{ea7} to \eqref{ea6} we obtain:
\begin{multline}
\label{ea8}
\int_{\Rr^N} \rho(x)[ u(x,t_2)- \tilde u(x,t_2)] \psi(x) \dd x -
\int_{\Rr^N}\rho(x)[ u(x, t_1) - \tilde u (x, t_1)] \psi(x) \dd x\\ 
= \int_{t_1}^{t_2} \int_{\Rr^N}\zeta (-\Delta)^{\sigma/2} \psi(x)
\dd x\,  \dd t,
\end{multline}
where $\zeta:= G(u)-G(\tilde u).$ 

For any $R>0$, let $\varphi_R$ be the function defined in \eqref{ea2} and \eqref{ea5}.
Formula \eqref{ea8} for $\psi=\varphi_R$ gives:
\begin{multline}
\label{ea9}
\int_{\Rr^N} \rho(x)[ u(x,t_2)- \tilde u(x,t_2)] \varphi_R(x) \dd
x - \int_{\Rr^N}\rho(x)[ u(x, t_1) - \tilde u (x, t_1)]
\varphi_R(x) \dd x \\ = \int_{t_1}^{t_2} \int_{\Rr^N}\zeta
(-\Delta)^{\sigma/2} \varphi_R(x) \dd x\,  \dd t,
\end{multline}
Now we estimate the absolute value of the right hand side of
\eqref{ea9}. In view of  
\eqref{ea15}, by using H\"older inequality we obtain:
\begin{align}
\left|\int_{t_1}^{t_2} \int_{\Rr^N} \zeta (-\Delta)^{\sigma/2} \varphi_R(x) \dd x \,\dd t \right| 
  \nonumber\\
\leq C \int_{t_1}^{t_2} \int_{\Rr^N} |u -\tilde u|^{1/q} \big|(-\Delta)^{\sigma/2} \varphi_R(x)\big| \dd x \,\dd t \nonumber \\ 
= C \int_{t_1}^{t_2} \int_{\Rr^N} [\rho(x)]^{-1/q}\big| (-\Delta)^{\sigma/2} \varphi_R(x) \big| |u-\tilde u|^{1/q}  [\rho(x)]^{1/q}\dd x \,\dd t \nonumber \\ 
\leq C  (t_2-t_1)^{1/p}\left( \int_{\Rr^N} [\rho(x)]^{-p/q}\big| (-\Delta)^{\sigma/2} \varphi_R(x) \big|^p \dd x\right)^{1/p}  \nonumber \\ 
 \cdot\left(\int_{t_1}^{t_2} \int_{\Rr^N}|u-\tilde u| \rho(x) \dd x \,\dd t \right)^{1/q} \label{ea10}. 
\end{align}
Set
\[
C_1:= (\max_{\overline{B}_{\hat R}} \rho)^{-p/q}, \qquad  C_2:= \int_{B_{\hat R}} \big| (-\Delta)^{\sigma/2}\varphi_1(x) \big|^p \dd x.
\]
Performing the change of variable $y:=\frac{|x|}R$, using
hypothesis \eqref{A1}, and properties \eqref{ea3} 
and \eqref{ea4} we
get:
\begin{align}
\label{ea11}
& \, \int_{\Rr^N} [\rho(x)]^{-p/q}\big|(-\Delta)^{\sigma/2} \varphi_R(x) \big|^p  \dd x  \nonumber \\
 \leq &\,   C_1 \int_{B_{\hat R}}  \big|(-\Delta)^{\sigma/2} \varphi_R(x) \big|^p \dd x
+\frac{1}{\hat C^{p/q}} \int_{\Rr^N\backslash B_{\hat R}}  \big|(-\Delta)^{\sigma/2} \varphi_R(x) \big|^p |x|^{\alpha\frac p q} \dd x \nonumber\\
=  &\, C_1 R^{N-p\sigma} \int_{B_{\hat  R/ R}} \big|(-\Delta)^{\sigma/2} \varphi_1(y) \big|^p \dd y \nonumber\\
&\qquad \qquad \qquad\qquad\qquad\quad +\frac{R^{N-p\sigma+\alpha \frac{p}{q}}}{\hat C^{p/q}} \int_{\Rr^N\backslash 
B_{\hat  R/ R}}  \big|(-\Delta)^{\sigma/2} \varphi_1(y) \big|^p |y|^{\alpha\frac p q} \dd y \nonumber\\
\leq &\, C_1C_2 R^{N-p\sigma} + \frac{R^{N-p\sigma+\alpha \frac p q}}{\hat C^{p/q}}  \int_{\Rr^N}\frac 1{ (1+ |y|^{(N+\sigma)})^p |y|^{-\alpha \frac p q}}\dd y\,.
\end{align}
Since $0<\alpha<\sigma$, we can choose $p>1$ so big that $\alpha<\frac{p\sigma- N }{p-1}$. Thus, since $q=\frac p{p-1}$,
\begin{equation}\label{ea12}
\begin{split}
-p\sigma + N+\alpha \frac p q <0,  \\
(N+\sigma)p - \alpha \frac p q > N\,.
\end{split}
\end{equation}
From \eqref{ea12}, \eqref{ea11}, \eqref{ea10}, for $u, \tilde u \in L^1_\rho(\Rr^N)$, it follows that
\begin{equation}\label{ea13}
\int_{t_1}^{t_2}\int_{\Rr^N} \zeta (-\Delta)^{\sigma/2}\varphi_R(x)
\dd x \to 0\quad\textrm{as}\;\; R\to \infty\,.
\end{equation}
On the other hand, since $u, \tilde u \in L^1_\rho(\Rr^N),
0\leq \varphi_R\le 1$ and $\varphi_R(x)\to 1$ as $R\to\infty$ for
every $x\in \Rr^N$, by the dominated convergence theorem,
\begin{multline}
\label{ea14}
\lim_{R\to \infty}\left\{ \int_{\Rr^N} \rho(x)[ u(x,t_2)- \tilde
u(x,t_2)]\varphi_R \dd x - \right.\\
\left. \int_{\Rr^N}\rho(x)[ u(x, t_1) - \tilde u (x, t_1)]\varphi_R \dd x\right\}\\ = \int_{\Rr^N}
\rho(x)[ u(x,t_2)- \tilde u(x,t_2)]\dd x - \int_{\Rr^N}\rho(x)[
u(x, t_1) - \tilde u (x, t_1)]\dd x \,.
\end{multline}
As a consequence of \eqref{ea9}, \eqref{ea13} and \eqref{ea14} we
have:
$$\int_{\Rr^N} \rho(x)[ u(x,t_2)- \tilde u(x,t_2)]  \dd x - \int_{\Rr^N}\rho(x)[ u(x, t_1) - \tilde u (x, t_1)] \dd x \, = \, 0\,.$$
Since $u,\tilde u \in C\big([0,\infty);L^1_\rho(\Rr^N)\big)$ and $u(x,0)=\tilde u(x, 0) = u_0(x)$ for almost every $x \in \Rr^N$, we have, as $t_1\to 0^+$,
\[ 
\int_{\Rr^N}\rho(x)[u(x,t_2)-\tilde u(x, t_2)]\dd x =
\int_{\Rr^N}\rho(x)[u(x,0)-\tilde u(x, 0)]\dd x =0.
\]
This implies $u(x,t_2)= \tilde u(x, t_2)$, for almost every $x \in \Rr^N$, because $\rho>0$ in $\Rr^N$. Since
$t_2>0$ were arbitrary, the proof is complete. 
\end{proof}

\begin{remark}
\label{oss1a}
\begin{itemize}
\item[i.] For problem \eqref{16011}, uniqueness is proved in \cite{RV2} supposing \eqref{A1} with $\alpha \in (0, N]$. Furthermore, uniqueness for {\em very weak } solutions is established in \cite{P1} assuming \eqref{A1} with $\alpha\in (0, 2]$.
\item[ii.] In the present situation, if one wanted to weaken hypothesis \eqref{A1} and to apply the same arguments as above in order to show uniqueness,  one should replace $\varphi_1$ by a function satisfying a decay condition stronger than \eqref{ea4}. Unfortunately, as shown in \cite{BonfV}, the decay \eqref{ea4} is the minimal one can expect.
\end{itemize}
\end{remark}

\subsection{Conservation of mass}
By exploiting some arguments introduced in the proof of Theorem \ref{tuni}, we can prove the following property of solutions to problem \eqref{06111}. 
\begin{proposition}\label{propcm}
Let assumptions \eqref{A0}, \eqref{A1} be satisfied. Let $u$ be the bounded solution to problem \eqref{06111}. Then 
\begin{equation} \label{ea48}
\int_{\Rr^N} u(x,t) \rho(x) \dd x \,=\, \int_{\Rr^N} u_0(x) \rho(x) \dd x\quad \textrm{for any}\;\; t>0\,. 
\end{equation}
\end{proposition} 

\begin{proof} 
Keep the same notation as in the proof of Theorem \ref{tuni}. By \eqref{ea6} with $\psi=\varphi_R$,
\begin{equation}\label{ea23b}
\int_{\Rr^N} \rho(x)[ u(x, t_2)- u(x, t_1) ] \,\varphi_R\,  \dd x \,=\,
\int_{t_1}^{t_2}\int_{\Rr^N} G(u) (-\Delta)^{\sigma/2}  \varphi_R\, \dd x\,\dd t.
\end{equation}
From \eqref{ea23b} and \eqref{ea13} with $\zeta$ replaced by $G(u)$, we obtain
\begin{equation}\label{ea24}
\int_{\Rr^N} u(x, t_2) \rho(x)\dd x \,=\, \int_{\Rr^N} u(x, t_1) \rho(x)\dd x\,.
\end{equation}
Since $u\in C\big([0,\infty); L^1_{\rho}(\Rr^N)\big)$, as a consequence of \eqref{ea24}, by sending $t_1\to 0^+$ we obtain the conclusion. 
\end{proof}

\section{Fast decaying density}
\label{fast}
Let us assume now that $\rho(x)$ vanishes {\em fast} as $|x|$ diverges, that is:
\begin{equation}
\label{A2}
\tag{{\bf A}$_2$}
\begin{array}{c}
  \textrm{there exist }\check C>0, \check R>0 \textrm{ and }\alpha\in (\sigma, \infty)\textrm{ such that }\\
  \rho(x)\leq \check C  |x|^{- \alpha}\textrm{ for all } x\in \Rr^N\backslash B_{\check R}\,\,.
\end{array}
\end{equation}
Under hypothesis \eqref{A2} we shall prove existence of a solution to problem \eqref{06111} satisfying an extra condition at infinity. From this we will infer nonuniqueness of solutions to 
the same problem, if we do not specify extra conditions at infinity.  Instead, uniqueness is restored in the class of solutions satisfying a suitable decay estimate as $|x|\to\infty$.

\subsection{Preliminary results}
We shall use the following Proposition, which is shown in
\cite{Rubin}, and some consequences of it.

\begin{proposition}
\label{propR} Let  $N\geq 1, r > 1, \frac 2 r < \beta < N$. Let
\begin{equation}\label{ea50}
\beta -\frac{N}{r} < \nu< N\frac{r-1}r\,.
\end{equation}
Suppose that $f(|x|)|x|^{\nu}\in L^r(\Rr^N)$. Then there exists a
constant $C>0$ such that
\[
\left|  \int_{\Rr^N} \frac{f(|y|)}{|x-y|^{N-\beta}}\, \dd y
\right| \leq C \,\big\||x|^\nu f\big\|_{L^r(\Rr^N)}\, |x|^{\beta-\nu-\frac
N r}
\]
for almost every $x\in \Rr^N$.
\end{proposition}

From Proposition \ref{propR} we deduce next
\begin{corollary}\label{corR}
Let $N\geq 2.$ Let assumptions \eqref{A0}, \eqref{A2} be
satisfied. Then
\begin{equation}\label{ea58}
\int_{\Rr^N}\frac{\rho(y)}{|x-y|^{N-\sigma}} dy \,\to \,0\quad
\textrm{as}\;\; |x|\to \infty\,.
\end{equation}
More precisely, 
\begin{itemize}
\item[i.] If $N=2$, then, for some $C>0$, there holds:
\begin{equation}\label{ea52}
\int_{\Rr^2}\frac{\rho(y)}{|x-y|^{2-\sigma}} dy\leq C
|x|^{\sigma-\nu-\frac 2 r}\quad \textrm{for all}\;\; x\in \Rr^2 
\end{equation}
provided  $0<\nu<\sigma$ and     
\begin{equation}\label{ea54}
\max\left\{\frac 2{\sigma},\frac{2}{2-\nu},
\frac{2}{\alpha-\nu}\right\}<r<\frac{2}{\sigma-\nu}.
\end{equation}
\item[ii.] If $N\ge 3$, then, for some $C>0$, there holds:
\begin{equation}\label{ea55}
\int_{\Rr^N}\frac{\rho(y)}{|x-y|^{N-\sigma}} dy\leq C
|x|^{\sigma-\frac N r}\quad \textrm{for all}\;\; x\in \Rr^N,
\end{equation}
provided 
\begin{equation}\label{ea53}
\max\left\{\frac N{\alpha}, \frac 2{\sigma} \right\} < r < \frac
N{\sigma}\,.
\end{equation} 
\end{itemize}
\end{corollary}

\begin{proof}
Observe that by \eqref{A2},
\begin{equation}
\label{ea20}
0<\rho(x)\leq \tilde \rho(x) \quad \textrm{for all}\;\; x\in \Rr^N,
\end{equation}
where
\[
\tilde \rho(x):=
\begin{cases}
\max_{\overline B_{\check R}} \rho &\textrm{if }x\in B_{\check R}, \\ \\
\check C|x|^{-\alpha} & \textrm{if } x\in \Rr^N\backslash B_{\check R}.
\end{cases}
\]
$i.$  Let $N=2$. Take $0<\nu<\sigma$. Observe that $|x|^\nu \tilde
\rho(|x|)\in L^r(\Rr^2)$ whenever $r(\alpha-\nu)>2$. Furthermore,
we can find $r>1$ such that \eqref{ea54} is verified.
This permits to apply Proposition \ref{propR} (with
$\beta=\sigma$) to infer that  
\[
\int_{\Rr^2}\frac{\tilde \rho(y)}{|x-y|^{2-\sigma}}\, \dd y \leq C \,\bigl\|| |x|^\nu \tilde \rho\bigr\|_{L^r(\Rr^2)}\, |x|^{\beta-\nu-\frac
2 r}
\]
Then, by \eqref{ea20}, \eqref{ea52} is verified.

\smallskip

$ii.$ Now, let $N\geq 3.$ Clearly, $\tilde \rho\in L^r(\Rr^N)$ whenever $\alpha \,r >N$.
Furthermore,  we can select $r>1$
such that \eqref{ea53} is verified. This
combined with Proposition \ref{propR} (with $\beta=\sigma, \nu=0$)
implies
\[
\int_{\Rr^N}\frac{\tilde \rho(y)}{|x-y|^{N-\sigma}}\, \dd y \leq C \,\|| \tilde \rho\|_{L^r(\Rr^N)}\, |x|^{\beta-\frac
N r}
\]
for almost every $x\in \Rr^N$. Then, by \eqref{ea20}, \eqref{ea55} holds true.
This completes the proof.
\end{proof}

\bigskip
\begin{remark}\label{remR}
Let us note that the first and the second inequality in condition \eqref{ea50} in Proposition \ref{propR}
are only used to deduce that
$\int_{\Rr^N\setminus B_1}\frac{f(|y|)}{|y|^{N-\beta}} dy<\infty$, and $f\in L^1_{loc}(\Rr^N)$, respectively. 
Indeed, such conditions implies that for every $x\in \Rr^N, \;$ the singular integral $\int_{\Rr^N}\frac{f(|y|)}{|x-y|^{N-\beta}} dy$ converges.
So, if $f\in L_{loc}^1(\Rr^N)$ by hypothesis, the thesis of Proposition \ref{propR} remains true without requiring the second inequality in \eqref{ea50}.
\end{remark}

\medskip

For further purposes, let us discuss another application of
Proposition \ref{propR}, combined with Remark \ref{remR}. To be specific, we will take 
$r>1$ and
\begin{equation}\label{ea56}
\nu > N\frac{r-1}{r}\,.
\end{equation}
The role of this condition will be clear in the sequel (see Theorem \ref{texi3} below). 

\smallskip

\begin{corollary}\label{cor2R}
Let $N\geq 2.$ Let assumption \eqref{A0} be satisfied. Moreover,
suppose that
\begin{equation}
\label{A2b} \tag{{\bf A}$_2^*$}
\begin{array}{c}
  \textrm{there exist}\; \check C>0, \check R>0 \; \textrm{and}\; \alpha\in (N, \infty) \; \textrm{such that}\\
  \rho(x)\leq \check C  |x|^{- \alpha}\quad \textrm{for all}\;\, x\in \Rr^N\backslash B_{\check R}\,\,.
\end{array}
\end{equation}
Then \eqref{ea58} is satisfied. More precisely, for some $C>0$, we have:
\begin{equation}\label{ea60}
\int_{\Rr^N}\frac{\rho(y)}{|x-y|^{N-\sigma}} dy\leq C
|x|^{\sigma-\nu-\frac N r}\quad \textrm{for all}\;\; x\in \Rr^N,
\end{equation}
provided $\frac N2 (2-\sigma)<\nu<N$ and \begin{equation}\label{ea59}
\max\left\{\frac 2{\sigma},\frac N{\alpha-\nu}\right\}<r<\frac
N{N-\nu}\,.
\end{equation} 
\end{corollary}
Observe that \eqref{ea56} is guaranteed by \eqref{ea59}. Moreover, the conclusion of Corollary \ref{cor2R}, valid for $\alpha>N$, is similar to that of Corollary \ref{corR}, where we had $\alpha>\sigma$.

\begin{proof}
Note that
$|x|^{\nu}\tilde \rho\in L^r(\Rr^N)$ whenever $(\alpha-\nu)r>N$. Moreover, $\tilde \rho\in L^1_{loc}(\Rr^N)$; hence,  in view
of Remark \ref{remR}, in order to apply Proposition \ref{propR} (with
$\beta=\sigma$) when condition \eqref{ea56} is satisfied,  we have to find $r>1$ and $\nu>0$ such that
\[
\begin{cases}
r>\frac 2{\sigma}\\
r(\sigma-\nu)<N \\
 r(N-\nu)<N \\
 r(\alpha-\nu)>N.
\end{cases}
\]
Since $0<\sigma<N$, this is equivalent to
\begin{equation}
\label{ea57}
\begin{cases} 
r>\frac 2{\sigma}\\
r(N-\nu)<N \\
r(\alpha-\nu)>N.
\end{cases}
\end{equation}
If we assume that $\frac N2 (2-\sigma)<\nu<N$ and $\alpha>N$, then
we can find $r>1$ such that \eqref{ea59} is verified. 
This implies \eqref{ea57}, hence the thesis follows. 
\end{proof}

\subsection{Existence of solutions satisfying an extra condition at infinity}

Fix any $\tau\geq 0.$ Using Corollary \ref{corR} we shall prove the existence of a solution $u$ to problem \eqref{06111} satisfying at infinity 
the following extra condition:
\begin{multline}
\label{ea21}
U(x,t):= \int_\tau^t G\big(u(x,s) \big) \dd s \to 0 \quad \textrm{as } |x|\to \infty, \\ 
\textrm{uniformly with respect to}\; t>\tau\,,
\end{multline}
where the function $G$ is defined by \eqref{ea22}. Furthermore, we will compute the decay rate of $U(x,t)$ as $|x|\to\infty$, uniformly with respect to $t>\tau$.

This is the content of next 
\begin{theorem}\label{texi2}
Let $N\geq 2.$ Let assumptions \eqref{A0}, \eqref{A2} be satisfied. Then there exists a bounded solution $u$ to problem \eqref{06111} such that condition \eqref{ea21} holds. More precisely,
\begin{itemize}
\item[i.] if $N=2$, then, for some $C>0$, we have:
\begin{multline}
\label{ea63}
U(x,t) \leq 2\|u_0\|_{\infty} \int_{\Rr^2} \frac{\rho(y)}{|x-y|^{2-\sigma}}\, \dd y \leq   C |x|^{\sigma-\nu-\frac 2 r}\\ 
\textrm{for almost every}\;\; x\in \Rr^2\setminus B_{\bar R}, \, t>\tau,
\end{multline}
provided $\bar R>0, 0<\nu<\sigma$ and
\eqref{ea54} is verified. 

\item[ii.] If $N\ge 3$, then, for some $C>0$, we have:
\begin{multline}\label{ea64}
U(x,t) \leq 2\|u_0\|_{\infty} \int_{\Rr^N} \frac{\rho(y)}{|x-y|^{N-\sigma}}\, \dd y \leq C |x|^{\sigma-\frac N r}\\ 
\textrm{for almost every}\;\; x\in
\Rr^N\setminus B_{\bar R}, \, t>\tau,
\end{multline}
provided $\bar R>0$ and \eqref{ea53} is verified.
\end{itemize}
\end{theorem}

Clearly, \eqref{ea63} and \eqref{ea64} hold true in the whole $\Rr^N$;
however, they are not useful in $B_{\bar R}$. Indeed, while the right hand sides goes to infinity for $|x|\to 0$, we know that $U$ is bounded.

\begin{proof}
For every $R>0$ let $u_R$ be the unique solution to problem \eqref{02013}. Recall that, by Proposition \ref{03012},
\begin{equation}\label{ea28}
0\leq u_R \leq \|u_0\|_{\infty} \quad \textrm{in}\;\; B_R\times (0,\infty)\,.
\end{equation}
Arguing as in Theorem \ref{02011}, it can be proved that $\displaystyle u=\lim_{R\to \infty} u_R$ solves \eqref{06111}. It remains to show \eqref{ea21}.
Define  
\[
U_R(x,t):=\int_\tau^t G\big(u_R(x,s) \big) \dd s\quad (x\in \Rr^N, \, t>\tau)
\]
and notice that $U_R\to U$ as $R\to \infty$ in $\Rr^N\times (\tau,\infty)$, where $U$ in defined in \eqref{ea21}.
It is direct to see that for every $R>0, t>\tau$ the function $U_R(\cdot, t)$ is a solution to problem 
\[
  \begin{cases}
  (- \Delta)^{\frac{\sigma}{2}} U_R = \rho[u_0 - u_R(\cdot, t)] & x\in B_R\\ 
      U = 0 & x\in \partial B_R\,.
    \end{cases}
\]
So, 
\[U_R(x,t)= \int_{B_R} K_R(x,y) \rho(y) [u(y,\tau) - u_R(y,t)] \dd y\quad (x\in B_R, \, t>\tau),\]
where $K_R$ is the Green function for equation $(-\Delta)^{\sigma/2} U=0$ in $B_R$, completed with Dirichlet zero boundary conditions. 
It is easily seen that 
\begin{multline}
\label{ea25}
\lim_{R\to \infty} U_R(x,t) = \int_{\Rr^N} K(x,y) \rho(y) [u(y,\tau) - u(y,t)] \dd y\\ \textrm{for any}\;\;  x\in \Rr^N, t>\tau,
\end{multline}
where 
\[
K(x,y):=\frac 1 {|x-y|^{N-\sigma}} \qquad (x,y\in \Rr^N, \, x\neq y)
\]
is the {\it Riesz Kernel}. In fact, there hold 
\begin{equation}\label{ea26}
0<K_R(x,y) \leq K(x,y) \quad \textrm{for any}\;\; x,y\in B_R, \, x\neq y,
\end{equation}
and 
\begin{equation}\label{ea27}
\lim_{R\to\infty} K_R(x,y) = K(x,y)\quad \textrm{for any}\;\; x,y\in \Rr^N, \,x\neq y.
\end{equation}

From \eqref{A0}, \eqref{ea28} and \eqref{ea26}, we have for all $x,y\in B_R$, $x\neq y$, $t>\tau$, 
\begin{equation}
\label{ea29}  
\Big| K_R(x,y)\rho(y)[u_R(y,\tau)- u_R(y,t)] \Big| \leq 2 \|u_0\|_{\infty} K(x,y) \rho(y)\,. 
\end{equation}
Now, fix any $x\in \Rr^N$, $t>\tau$. Note that, in view of \eqref{A2}, the function $y\mapsto 2\|u_0\|_\infty K(x,y) \rho(y)$ belongs to $L^1(\Rr^N)$. Thus, from \eqref{ea27} and the 
dominated convergence theorem the limit \eqref{ea25} follows. Hence 
\begin{equation}\label{ea23}
U(x,t)= \int_{\Rr^N} K(x,y) \rho(y) [u(y, \tau) - u(y,t)] \dd y\qquad (x\in \Rr^N, \, t>\tau)\,.
\end{equation}

From \eqref{ea23} and Corollary \ref{corR} we get \eqref{ea63} and \eqref{ea64}, which in turn imply \eqref{ea21}. This completes the proof. 
\end{proof}

From the proof of Theorem \ref{texi2} and Corollary \ref{cor2R} we obtain the following result. 
\begin{theorem}\label{texi3}
Let assumptions of Theorem \ref{texi2} be satisfied with \eqref{A2} replaced by
\eqref{A2b}. The the conclusion of Theorem \ref{texi2} remains true, with \eqref{ea63} and \eqref{ea64} replaced by the following estimate: 
\begin{multline}\label{ea65}
U(x,t)\leq C |x|^{\sigma-\nu-\frac N r} \textrm{ for almost every } x\in \Rr^N\setminus B_{\bar R}, \, t>\tau, \text{ for some $C>0$, } \\
\text{ provided} \;\bar R>0, \frac N2 (2-\sigma)<\nu<N\; \textrm{and}\; \eqref{ea59} \;\textrm{is verified.} 
\end{multline}
\end{theorem}

Let us recall that solutions constructed in Theorem \ref{texi2} and \ref{texi3} are minimal. 

\subsection{Nonuniqueness of solutions}
From Theorem \ref{texi2} we infer next nonuniqueness of solutions to problem \eqref{06111}. 

\begin{corollary}\label{cornu1}
Let $N\geq 2$. Let assumptions \eqref{A0}, \eqref{A2} be satisfied. Let $\rho\in L^1(\Rr^N)$. Then, for $u_0\equiv c\in (0,\infty)$, problem \eqref{06111} admits at least two bounded solutions.
\end{corollary}

\begin{proof}
Note that since $\rho\in L^1(\Rr^N)$, $u_0\equiv c\in L^\infty(\Rr^N)\cap L^1_{\rho}(\Rr^N).$ Clearly, $\tilde u \equiv c$ is a solution 
to problem \eqref{06111}. Moreover, by Theorem \ref{texi2} there exists a solution $u$ to problem \eqref{06111} satisfying condition \eqref{ea23}, so $u\not\equiv \tilde u$.
This completes the proof. 
\end{proof}  

\bigskip

Let us mention that if we only assume $u_0\in L^{\infty}(\Rr^N)$, then it is easy to verify  (see, $e.g.$ \cite{PT1}) that problem \eqref{06111} admits a {\em very weak} bounded solution $u$, in the sense that 
$u\in L^\infty(\Rr^N\times(0,\infty))$ satisfies  
\[
\int_0^T \int_{\Rr^N} \rho\, u\, \partial_t \psi \dd x\, \dd t =  - \int_0^T \int_{\Rr^N}u^m\,
(-\Delta)^{\sigma/2}\psi \, \dd x\, \dd t\,
\]
for any $T>0$ and any test function $\psi\in C_0^\infty(\Rr^N\times [0,T])$.
Thus next nonuniqueness result for very weak solutions immediately follows, without supposing $\rho\in L^1(\Rr^N)$. 
\begin{corollary}\label{cornu2}
Let $N\geq 2.$ Let assumptions \eqref{A0}, \eqref{A2} be satisfied. Then, for $u_0\equiv c\in (0,\infty)$, problem \eqref{06111} admits at least two very weak solutions.
\end{corollary}

\subsection{Uniqueness of solutions satisfying a decay estimate at infinity}
\begin{theorem}\label{tunc} Let $N\geq 1$.
Let assumptions \eqref{A0}, \eqref{A2b} be satisfied. Let $\tilde
u$ be the minimal solution to problem \eqref{06111}, let $u$ be
any solution to problem \eqref{06111} such that 
\eqref{ea65} is satisfied with $\tau=0$. Then $u\equiv\tilde 
u$.
\end{theorem}

Note that the minimal solution $\tilde u$ is that constructed in Theorem \ref{texi3}. 

\begin{proof}
Repeat the proof of Theorem \ref{tuni}. We obtain:
\begin{multline}
\label{ea66} \int_{\Rr^N} \rho(x)[ u(x,t_2)- \tilde u(x,t_2)]
\varphi_R(x) \dd x - \int_{\Rr^N}\rho(x)[ u(x, t_1) - \tilde u (x,
t_1)] \varphi_R(x) \dd x \\ = \int_{t_1}^{t_2} \int_{\Rr^N}\zeta
(-\Delta)^{\sigma/2} \varphi_R(x) \dd x\,  \dd s,
\end{multline}
Now we estimate the absolute value of the right hand side of
\eqref{ea66}. Let $\gamma>0$ be a constant to be fixed later; put $\tilde C:= [2T \|u_0\|_{\infty}^m]^{1-\frac1q}.$
Define 
\[ \xi(x):=\begin{cases}
1 & |x|\leq 1 \\
|x|^{-\gamma} & |x|>1\,.
\end{cases}
\]
In view of \eqref{ea3}, \eqref{ea15}, and by using H\"older
inequality and the fact that $\tilde u$ is minimal we obtain:
\begin{gather}
\left|\int_{t_1}^{t_2} \int_{\Rr^N} \zeta (-\Delta)^{\sigma/2}
\varphi_R(x) \dd x \,\dd s \right| \nonumber \\
\leq \tilde C \int_{t_1}^{t_2} \int_{\Rr^N} [\xi(x)]^{1/q}\big| (-\Delta)^{\sigma/2} \varphi_R(x) \zeta^{1/q} [\xi(x)]^{-1/q}\dd x \,\dd s \nonumber \\
\leq \tilde C (t_2-t_1)^{1/p}\left( \int_{\Rr^N} [\xi(x)]^{\frac{p}q}\big| (-\Delta)^{\sigma/2} \varphi_R(x) \big|^p \dd x\right)^{1/p} \nonumber \\
\qquad\qquad\qquad\cdot
	\left[ \int_{\Rr^N} \left( \int_{t_1}^{t_2} G\big(u(x,s)\big)\dd s\right)  [\xi(x)]^{-1} \dd x \right]^{1/q}. 
\label{ea67}
\end{gather}
Let $t_2=t.$ Hence, for $\tau=0<t_1$ and $G(u)\geq 0$, 
\begin{equation}\label{ea77}
\int_{t_1}^t G(u(x,s))\dd s \leq U(x,t)\quad (x\in \Rr^N, t>0)\,.
\end{equation}

Observe that, from hypothesis \eqref{ea65} we get, 
\[\int_{\Rr^N} U(x,t) \xi(x) \dd x = \int_{B_1} U(x,t)  \dd x+ \int_{\Rr^N\setminus B_1} U(x,t)\xi(x) \dd x
\]
\[\leq T\|u_0\|_{\infty}^m|B_1| +C\int_{\Rr^N\setminus B_1} \frac{\dd x}{|x|^{-\sigma +\nu +\frac{N}r+\gamma}}\,.
\]
In view of \eqref{ea59} there holds \eqref{ea56}, so we can select $\gamma>0$ such that
\begin{equation}\label{ea73}
N-\frac N r -\nu +\sigma < \gamma <\sigma\,.
\end{equation}
The first inequality in \eqref{ea73}  implies that 
\[ \int_{\Rr^N\setminus B_1} \frac{\dd x}{|x|^{-\sigma +\nu +\frac{N}r+\gamma}}<\infty , \]
so 
\begin{equation}\label{ea74}
\int_{\Rr^N} U(x,t) \xi(x) \dd x <\infty\,. 
\end{equation}

\smallskip

We need the second inequality of \eqref{ea73} in the sequel. In fact, set
\[
C_1:=
\int_{B_{1}} \big| (-\Delta)^{\sigma/2}\varphi_1(x) \big|^p
\dd x.
\]
Arguing as in the proof of Theorem \ref{tuni}
we get:
\begin{align}
\label{ea70}
 &\,    \, \int_{\Rr^N} \big|(-\Delta)^{\sigma/2} \varphi_R(x) \big|^p |x|^{-\frac{\gamma p}q} \dd x  \nonumber \\
 \leq &\, \quad C_1 R^{N-p\sigma} + R^{N-p\sigma+\gamma
\frac p q} \int_{\Rr^N\backslash B_{1}}\frac
1{ (1+ |y|^{(N+p)\sigma})|y|^{-\gamma \frac p q}}\dd y\,.
\end{align}
Since $0<\gamma<\sigma$, we can choose $p>1$ so big that
$\gamma<\frac{p\sigma- N }{p-1}$. So, since $q=\frac p{p-1}$,
\begin{equation}\label{ea71}
\begin{split}
-p\sigma + N+\alpha \frac p q <0, \\
(N+\sigma)p - \gamma \frac p q > N\,.
\end{split}
\end{equation}
From \eqref{ea71}, \eqref{ea70}, \eqref{ea74}, \eqref{ea77} and \eqref{ea67} we can infer that 
\begin{equation}\label{ea75}
\int_{\tau}^{t} \int_{\Rr^N} \zeta (-\Delta)^{\sigma/2}
\varphi_R(x) \dd x \,\dd s\to 0\quad \textrm{as}\;\; R\to \infty\,.
\end{equation}

\smallskip

On the other hand, since $u, \tilde u \in L^1_\rho(\Rr^N), 0\leq
\varphi_R\le 1$ and $\varphi_R(x)\to 1$ as $R\to\infty$ for every
$x\in \Rr^N$, by the dominated convergence theorem,
\begin{multline}
\label{ea76} \lim_{R\to \infty}\left\{ \int_{\Rr^N} \rho(x)[
u(x,t)- \tilde u(x,t)]\varphi_R \dd x - \int_{\Rr^N}\rho(x)[
u(x, t_1) - \tilde u (x,t_1)]\varphi_R \dd x\right\}\\ =
\int_{\Rr^N} \rho(x)[ u(x,t)- \tilde u(x,t)]\dd x -
\int_{\Rr^N}\rho(x)[ u(x, t_1) - \tilde u (x, t_1)]\dd x \,.
\end{multline}
As a consequence of \eqref{ea66}, \eqref{ea75} and \eqref{ea76} we
have:
$$\int_{\Rr^N} \rho(x)[ u(x,t)- \tilde u(x,t)]  \dd x - \int_{\Rr^N}\rho(x)[ u(x, t_1) - \tilde u (x, t_1)] \dd x \, = \, 0\,.$$
Since $u,\tilde u \in C\big([0,\infty);L^1_\rho(\Rr^N)\big)$ and
$u(x,0)=\tilde u(x, 0) = u_0(x)$ for almost every $x\in \Rr^N$, we have, as $t_1 \to 0^+$,
\[
\int_{\Rr^N}\rho(x)[u(x,t)-\tilde u(x, t)]\dd x =
\int_{\Rr^N}\rho(x)[u(x,0)-\tilde u(x, 0)]\dd x =0.
\]
This implies $u(x,t)= \tilde u(x, t)$, for almost every $x \in \Rr^N$,
because $\rho>0$ in $\Rr^N$. Since $t>0$ were arbitrary, the
proof is complete.
\end{proof}

\section{General initial data for the half-Laplacian}
\label{half}
We present now some results for problem \eqref{06111} in the particular case $\sigma=1$, that is
\begin{equation}
\label{ea30}
    \begin{cases}
   \rho\, \partial_t u + (- \Delta)^{\frac{1}{2}}\left[ u^m\right] = 0  & x\in \Rr^N, \, t>0\\
     u=u_0 & x\in \Rr^N, \, t=0\,.
    \end{cases}
\end{equation}
More precisely, we revisit here some of the  results in the previous Sections, removing the assumption $u_0\in L^\infty(\Rr^N)$. 
Preliminarily, in the following Subsection we establish a smoothing estimate for solutions of problem \eqref{ea30}.

\subsection{Smoothing effect}
\label{smoothing}
The proof of the smoothing estimate will make use of the following

\begin{lemma}\label{lemmas}
Let assumption  \eqref{A0}{\em (i)--(iii)} be satisfied and $\sigma=1$.
Let $u$ be a solution to problem \eqref{ea30}. Then 
$\partial_t u \in L^1_{loc}\big((0,\infty); L^1_{\rho}(\Rr^N)\big)$, and 
\begin{equation}\label{ea31}
\rho\, \partial_t u \geq -  \,\frac{\rho u}{(m-1)t}\quad \textrm{in } \Rr^N\times (0,\infty)
\end{equation}
in the sense of distributions.  Furthermore, for any $0\leq \tau<T$,
\begin{multline}\label{ea41}
\int_{\tau}^T\int_{\Rr^N} |(-\Delta)^{\frac{1}4}u^m| \, \dd x \, \dd t +\frac 1{m+1}\int_{\Rr^N}u^{m+1}\rho\, \dd x \\
= \frac 1{m+1}\int_{\Rr^n}\rho \, u^{m+1}(x,\tau)\, \dd x.
\end{multline}
\end{lemma}
Lemma \ref{lemmas} can be proved by applying the same arguments as in \cite[Proposition 3.1, Theorem 5.4]{V1} and \cite[Section 8]{V2}; we omit the details. 

\smallskip

We state the smoothing estimate in the following
\begin{theorem}\label{tse}
Let assumption \eqref{A0}{\em (i)--(iii)} be satisfied and $\sigma=1$;
suppose $\rho\in L^\infty(\Rr^N)$. Let $u$ be a solution to problem \eqref{ea30}. Then there exists a positive constant $C=C(N, m)$ such that 
\begin{equation}\label{ea32}
\|u(\cdot, t)\|_{L^\infty(\Rr^N)}\,\leq\, C\, t^{-\theta} \, \|u_0\|^{\theta/N}_{L^1_\rho(\Rr^N)}\quad \textrm{for any}\;\; t>0\,,
\end{equation}
where $\theta=\frac 1{m-1+\frac 1 N}.$
\end{theorem}
Theorem \ref{tse} can be proved by minor changes in the proof of \cite[Theorem 2.4]{V1}. However, we sketch the proof for reader's convenience. 

\begin{proof} We can consider the solution $(u,w)$ to problem \eqref{06112} with $\sigma=1$. Recall that, according with notations used in Section \ref{existence}, 
when dealing with harmonic extension, we use the notations $\Omega:= \Rr^{N+1}_+$ and $\Gamma:= \bar \Omega\cap \{y=0\}$. 

In view of Lemma \ref{lemmas}, from equality \eqref{291101} we deduce:
\[ 
\int_{\Omega} \langle \nabla w, \nabla \psi  \rangle \,\dd x \,\dd y +  \int_{\Gamma} \rho\, \partial_t u \,\psi\, \dd x = 0. 
\] 
This combined with \eqref{ea31} yields
\[
\int_{\Omega} \langle \nabla w, \nabla \psi  \rangle \,\dd x \,\dd y - L_t \int_{\Gamma} \rho\, u \,\psi \,\dd x \leq  0, 
\] 
where $L_t:=\frac 1{(m-1)t}, \psi\geq 0$. In view of Lemma \ref{lemmas}, $(u,w)$ is also a so-called {\it strong solution}, that is it solves problem \eqref{06112} with $\sigma=1$ almost everywhere in $\Omega\times (0,\infty)$ (see \cite[Section 5.3]{V1}). So, we can choose $\psi=w^p, p>0$ (see \cite[p. 187]{Vlibro}). Easy computations and the trace embedding give, for some constant $\tilde C>0$,
\begin{gather}
L_t \int_{\Gamma} \rho \,u^{mp +1}\, \dd x \, \geq\,  \int_{\Omega} \langle \nabla w, \nabla w^p \rangle\,  \dd x \dd y \nonumber\\
= \frac{4p}{(p+1)^2}\int_{\Omega}\big|\nabla w^{\frac{p+1}2}\big|^2\, \dd x\, \dd y \, \geq\,   \tilde C \frac{4p}{(p+1)^2}\left(\int_{\Gamma} u^\frac{Nm(p+1)}{N-1} \dd x\right)^\frac{N-1}{N}\nonumber\\
\geq \tilde C\,  \bar C \, \left(\int_{\Gamma}u^{\frac{Nm(p+1)}{N-1}}\rho\,  \dd x \right)^{\frac{N-1}N}, \label{30011}
\end{gather}
where $\bar C:= \frac{4p}{(p+1)^2}\|\rho\|_{\infty}^{-\frac{N-1}{N}}$.

\smallskip

For any $t>0$ and $q\in [1,\infty)$, set $\|u(\cdot, t)\|_{L^q_{\rho}(\Gamma)}\equiv \|u\|_{q}:= \left(\int_\Gamma |u|^q \rho\, \dd x \right)^{\frac 1 q}$. Recall that, in view of assumption \eqref{A0}-(i), the measure $\rho(x)\dd x$ is absolutely continuous with respect the Lebesgue measure in $\Rr^N$, so $\displaystyle\lim_{q\to\infty}\|u\|_q=\|u\|_{\infty}$, where $\|\cdot\|_{\infty}$ is the usual norm in $L^\infty(\Rr^N)$. 

Inequality \eqref{30011} imply:
\begin{equation}\label{ea33}
\|u\|_{sm(p+1)}\leq\left( \frac{L_t}{\tilde C\bar C}\right)^{\frac 1{m(p+1)}} \|u\|_{mp+1}^{\frac{mp+1}{m(p+1)}}\,,
\end{equation}
where $s:=\frac N{N-1}\,.$ Iterating \eqref{ea33} we get:
\begin{equation}\label{ea34}
\|u\|_{q_{k + 1}} \leq \left( \frac{L_t}{\tilde C\bar C}\right)^{\frac s{q_{k+1}}} \|u \|_{q_k}^{\frac{sq_k}{q_{k+1}}},
\end{equation}
where $q_{k+1}:=s(q_k+m-1), q_0:=mp+1\,.$

\smallskip

Easy computations  (see \cite[Theorem 2.4]{V1} for details) give: 
\begin{equation}\label{ea35}
\| u\|_{\infty} \leq \left( \frac{L_t}{\tilde C\bar C}\right)^{\frac N A} \|u\|_{q_0}^{\frac{q_0}A},
\end{equation}
where $A:= q_0+N(m-1)\,.$

Choose $p=\frac 1 m$; consequently $q_0=2$,  $A=2+N(m-1)$ and $\bar C=\frac{4m}{m^2+2m+1}$. By \eqref{ea35}, since $m\geq 1$,
\[
\| u\|_{2}\leq \| u\|_1^{\frac 1 2}\|u\|_{\infty}^{\frac 1 2} \leq \left( \frac{L_t}{\tilde C\bar C}\right)^{\frac{N}{2mA}}\|u\|_1^{\frac 1 2}\|u\|_{2}^{\frac 1{mA}},
\]
so
\begin{equation}\label{ea36}
\|u\|_2\leq  \left( \frac{L_t}{\tilde C\bar C}\right)^{\frac N{2(1+N(m-1))}}\|u\|_1^{\frac{2+N(m-1)}{2(1+N(m-1))}}\,.
\end{equation}
From \eqref{ea36}, \eqref{ea35} and \eqref{ea40bb} we obtain:
\[
\|u\|_{\infty}\leq \left( \frac{L_t}{\tilde C\bar C}\right)^\theta
\|u\|_1^{\theta/N}
\leq  C t^{-\theta}
\|u_0\|_1^{\theta/N}.
\qedhere
\]
\end{proof}

\begin{remark}
\label{oss3}
In \cite{V2} the smoothing estimate has been established for problem \eqref{16012} for $\sigma\in (0,2)$. It remains to be understood whether such effect holds for the general problem \eqref{06111} for any $\sigma\in (0,2)$.  
\end{remark}

\subsection{Existence of solutions}\label{exiu}
We shall assume, instead of \eqref{A0}, the following
\begin{equation}
\label{A01} 
\tag{{\bf A}$_0^*$}
\begin{cases}
\text{(i)} &\rho\in C(\Rr^N)\cap L^{\infty}(\Rr^N), \,\rho>0 \text{ in } \Rr^N;\\
\text{(ii)}& u_0\in L^+_\rho(\Rr^N)\\
\text{(iii)} &m\ge 1\,.
\end{cases}
\end{equation}
Note that now $u_0$ is not necessary bounded, on the other side we require that $\rho$ is so.
Concerning existence of solutions, we shall prove next
\begin{theorem}\label{teu}
Let assumption \eqref{A01} be satisfied. Then there exists a minimal solution to problem \eqref{ea30}. Furthermore, \eqref{ea40bb} holds true.
\end{theorem}

The same arguments used to prove Proposition \ref{propcontraz} yield the following
\begin{proposition}\label{propcontrazu}
Let assumption \eqref{A01} be satisfied. Let $(u,w)$ and $(\hat u, \hat w)$ be minimal solutions provided by Theorem \ref{teu}, corresponding to initial data $u_0$ and $\hat u_0$, respectively. Then, for any $t>0$, inequality \eqref{e40b} holds true.
\end{proposition}

\begin{proof}[Proof of Theorem \ref{teu}]
Let $\{u_{0n}\}\subset L^1_{\rho}(\Rr^N)\cap L^\infty(\Rr^N)$ such that $u_{0n}\geq 0$, $u_{0n}\to u_0$ in $L^1_{\rho}(\Rr^N)$ as $n\to \infty.$
For any $n\in \mathbb  N$, let $u_n$ be the minimal solution provided by Theorem \ref{02011} corresponding to initial datum $u_{0n}$. In view of \eqref{e40b}, for any $t>0$, $u_n(\cdot, t)\to u(\cdot, t)$ in 
$L^1_{\rho}(\Rr^N)$ as $n\to\infty$, for some function $u$. Moreover, by standard results in nonlinear semigroup theory, $u_n\to u$ in $C\big([0,\infty); L^1_{\rho}(\Rr^N) \big)$.

Take any $\tau>0$. By \eqref{ea32}, there exists $C_{\tau}>0$ such that for all $n\in \mathbb N$
\begin{equation}\label{ea42}
0\leq u_n\leq C_{\tau}\qquad \textrm{in }\Rr^N\times (\tau,\infty).
\end{equation} 
We can find a positive constant $C_1=C_1(\|u_0\|_{L^1_\rho})$ such that $\|u_{0n}\|_{L^1_{\rho}}\leq C_1$ for all $n\in \mathbb N$. By Theorem \ref{02011}, 
$\|u_n(\cdot, \tau)\|_{L^1_{\rho}}\leq C_1$. This combined with \eqref{ea42} implies that there exists $C_2>0$ such that 
\begin{equation}\label{ea43}
\|u_n(\cdot, \tau)\|_{L^m_{\rho}}\leq C_2.
\end{equation}
From \eqref{ea41} and \eqref{ea43} we deduce that 
\begin{equation}\label{ea44}
\int_0^{\infty}\|u_n^m\|^2_{\dot H^{1/2}} \, \dd t\leq \frac 1{m+1}\, C_2^m.
\end{equation}
From \eqref{ea42} and \eqref{ea44}, since $\tau>0$ was arbitrary, we can infer that $u$ solves equation
\[ 
\rho\, \partial_t u + (-\Delta)^{1/2} [u^m]=0\qquad (x\in \Rr^N, \,t>0).
\]   
Observe that  for all $t>0$
\begin{multline*}
\int_{\Rr^N} |u(x,t)- u_0(x)|\, \rho(x) \, \dd x\leq \int_{\Rr^N} |u(x,t)- u_n(x,t)|\, \rho(x)\,  \dd x \\
+\int_{\Rr^N} |u_n(x,t)-u_{0n}(x)|\, \rho(x) \, \dd x +\int_{\Rr^N}|u_{0n}(x)- u_0(x)|\, \rho(x)\,  \dd x.
\end{multline*}
In view of the contraction principle in $L^1_{\rho}$ and continuity in $L^1_{\rho}$, this implies that $u(x,0)=u_0(x)$ for almost every $x\in \Rr^N$. This completes the proof.  
\end{proof}
\label{subsec:halfunb}

\subsection{Slowly decaying density}
\label{slowu}
Concerning uniqueness of solutions we shall prove next
\begin{theorem}\label{tuniu} 
Let $\sigma=1$ and assumptions \eqref{A01}, \eqref{A1} be satisfied. Then problem \eqref{ea30} admits at most one solution.
\end{theorem}

\begin{proof} Suppose, by contradiction, that there exist two different solutions $u, \hat u$. Then, for some $T>0$ and $\epsilon>0$, there holds
\begin{equation}\label{ea45}
\| u(\cdot, T) -\hat u(\cdot, T)\|_{L^1_{\rho}}>\epsilon\,.
\end{equation}
Since $u, \hat u\in C\big([0,\infty); L^1_\rho(\Rr^N)\big)$, we can select $0<\tau<T$ such that 
\[\|u(\cdot, \tau)  - u_0 \|_{L^1_{\rho}}<\frac{\epsilon}{4}\,\text{ and }\,\|\hat u(\cdot, \tau)  - u_0 \|_{L^1_{\rho}}<\frac{\epsilon}{4}\,,
\]
thus
\begin{equation}\label{ea46}
\|\hat u(\cdot, \tau)-u(\cdot, \tau)\|_{L^1_\rho}<\frac{\epsilon}2\,.
\end{equation}
Let $v$ be the solution to problem 
\[
    \begin{cases}
   \rho\, \partial_t v + (- \Delta)^{\frac{1}{2}}\left[  v^m\right] = 0  & x\in \Rr^N, \, t>0\\
     v= u(x,\tau) & x\in \Rr^N, \, t=\tau\,,
    \end{cases}
\]
and $\hat v$ the solution to problem 
\[
    \begin{cases}
   \rho\, \partial_t \hat v + (- \Delta)^{\frac{1}{2}}\left[ \hat v^m\right] = 0  & x\in \Rr^N, \, t>0\\
     \hat v= \hat u(x,\tau) & x\in \Rr^N, \, t=\tau\,.
    \end{cases}
\]
Note that, in view of \eqref{ea32}, $u, \hat u $ are bounded in $\Rr^N\times [\tau,\infty)$. Hence such solutions $v,\hat v$, provided by Theorem \ref{02011},  are bounded as well. 
By Proposition \ref{propcontraz} and \eqref{ea46}, for all $t>\tau$,
\[
\| \hat v(\cdot, t)-v(\cdot, t)\|_{L^1_\rho}\leq \|\hat u(\cdot, \tau)-u(\cdot, \tau) \|_{L^1_\rho} <\frac{\epsilon}2.
\]
In view of \eqref{A1} for $u, v, \hat u, \hat  v\in L^{\infty}(\Rr^N\times (\tau,\infty))$,  by Theorem \ref{tuni}, $v=u$, $\hat v=\hat u$. Hence 
\begin{equation}\label{ea47}
\|\hat u(\cdot, t) - u(\cdot, t)\|_{L^1_\rho}<\frac{\epsilon}2\,.
\end{equation}
If we choose $t=T$, \eqref{ea47} is in contradiction with \eqref{ea45}. Hence $u\equiv \hat u$. The proof is completed. 
\end{proof}

Conservation of mass property remains true even if we assume \eqref{A01} instead of \eqref{A0}. In fact, we have next
\begin{proposition}
\label{propcmu}
Let $\sigma=1$ and assumptions \eqref{A01}, \eqref{A1} be satisfied. Then equality \eqref{ea48} holds true. 
\end{proposition}
\begin{proof}
The unique solution $u$ is obtained as described in the proof of Theorem \ref{teu}. By Proposition \ref{propcm}, for every $n\in \mathbb N$,
\[\int_{\Rr^N} u_n(x,t)\rho(x) \dd x = \int_{\Rr^N} u_{0n}(x) \rho(x) \dd x\,.\]
Letting $n\to\infty$, since $u_n(\cdot, t)\to u(\cdot, t) \in L^1_{\rho}(\Rr^N)$ for any $t>0$ and $u_{0n}\to u_0$ in $L^1_\rho$ as $n\to\infty$, the thesis follows.
\end{proof}

\subsection{Fast decaying density}
\label{fastu}

\begin{theorem}\label{texiu}
Let $N\geq 2$, $\tau>0, \sigma=1$. Let assumptions \eqref{A01}, \eqref{A2} be satisfied. Then there exists a solution $u$ to problem \eqref{ea30} such that condition \eqref{ea21} is satisfied. More precisely, inequalities \eqref{ea63} and \eqref{ea64} with $\sigma=1$ hold true. 
\end{theorem}

\begin{proof}
For any $n\in \mathbb N$, let $u_n$ be the minimal solution constructed in the proof of Theorem \ref{teu}. For any $n\in\mathbb N$, define $U_n(x,t):=\int_{\tau}^t G(u_n(x,s))ds,\,\,x\in \Rr^N, t>\tau.$
Repeating the proof of Theorem \ref{texi2}, we get for any $n\in \mathbb N$
\begin{multline*}
0\leq U_n(x,t)\leq\int_{\Rr^N}\left|K(x,y) \rho(y)\big[u_n(y,\tau)-u_n(y,t) \big]\dd y  \right|\\
 \textrm{for all } x\in\Rr^N, t>\tau.
 \end{multline*}
By \eqref{ea32}, there exists $C>0$ such that $\|u_n\|_{L^\infty(\Rr^N\times (\tau,\infty))}\leq C$ uniformly with respect to $n$. So, 
\[ 
0\leq U_n(x,t) \leq 2C\int_{\Rr^N} K(x,y)\rho(y) \dd y \quad \textrm{for all }x\in\Rr^N, \, t>\tau, \, n\in \mathbb N. 
\]
Sending $n\to \infty$, this yields
\[0\leq U(x,t) \leq 2C\int_{\Rr^N} K(x,y)\rho(y) \dd y \quad \textrm{for all }x\in\Rr^N, \, t>\tau.
\]
Hence the conclusion follows as well as in the proof of Theorem \ref{texi2}.  
\end{proof}

\begin{theorem}\label{tunc2}
Let $N\geq 2, \sigma=1$. Let assumptions \eqref{A01}, \eqref{A2b} be satisfied. Let $\tilde
u$ be the minimal solution to problem \eqref{ea30}, let $u$ be
any solution to problem \eqref{ea30} such that 
\eqref{ea65} is satisfied with $\tau>0$ and $\sigma=1$. Then $u\equiv\tilde 
u$.
\end{theorem}

\begin{proof} Suppose, by contradiction, that $u\neq \hat u$. Then, for some $T>0$ and $\epsilon>0$, there holds
\eqref{ea45}. Since $u, \hat u\in C\big([0,\infty); L^1_\rho(\Rr^N)\big)$, we can select $0<\tau<T$ such that  \eqref{ea46} is verified.
Let $v$ be the minimal solution to problem 
\begin{equation}\label{ea80}
    \begin{cases}
   \rho\, \partial_t v + (- \Delta)^{\frac{1}{2}}\left[  v^m\right] = 0  & x\in \Rr^N, \, t>0\\
     v= u(x,\tau) & x\in \Rr^N, \, t=\tau\,,
    \end{cases}
\end{equation}
and $\hat v$ the minimal solution to problem 
\begin{equation}\label{ea81}
    \begin{cases}
   \rho\, \partial_t \hat v + (- \Delta)^{\frac{1}{2}}\left[ \hat v^m\right] = 0  & x\in \Rr^N, \, t>0\\
     \hat v= \hat u(x,\tau) & x\in \Rr^N, \, t=\tau\,;
    \end{cases}
\end{equation}
Such solutions $v,\hat v$ are provided by Theorem \ref{texi3},  hence \eqref{ea65} holds true with $U$ replaced by $\int_{\tau}^t G\big(v(x,s) \big)\dd s$ or 
$\int_{\tau}^t G\big(\hat v(x,s) \big)\dd s$, with $\sigma=1$. Hence, by Theorem \ref{tunc}, $v=u$, $\hat v=\hat u$ (note that in Theorem \ref{tunc} we had $\tau=0$, however now we can apply it for $\tau>0$, since in problems \eqref{ea80},\eqref{ea81} initial conditions are given for $t=\tau>0$).
 
By Proposition \ref{propcontraz} and \eqref{ea46}, for all $t>\tau$,
\[
\| \hat u(\cdot, t)-u (\cdot, t)\|_{L^1_\rho}\leq \|\hat u(\cdot, \tau)-u(\cdot, \tau) \|_{L^1_\rho} <\frac{\epsilon}2.
\]
If we choose $t=T$, \eqref{ea47} is in contradiction with \eqref{ea45}. Hence $u\equiv \hat u$. The proof is completed. 
\end{proof}

\begin{remark}
Note that arguments in Subsections \ref{exiu}-\ref{fastu} only use $\sigma=1$ and $\rho\in L^\infty(\Rr^N)$ to apply \eqref{ea32}.  Hence all results in these Subsections remain true for $0<\sigma<2$ and general $\rho\in C(\Rr^N)$, provided a smoothing estimate like \eqref{ea32} can be established for these values of the parameter $\sigma$ and for such a density $\rho$.
\end{remark}

\bibliographystyle{plain}

\end{document}